\documentclass[a4paper,leqno]{amsart}
\usepackage{latexsym}
\usepackage[english]{babel}
\usepackage{fancyhdr}
\usepackage[mathscr]{eucal}
\usepackage{amsmath}
\usepackage{mathrsfs}
\usepackage{mathabx}
\usepackage{amsthm}
\usepackage{amssymb}
\usepackage{amscd}
\usepackage{bbm}
\usepackage{graphicx}
\usepackage{graphics}
\usepackage{latexsym}
\usepackage{color}
\usepackage{xcolor}
\usepackage{enumitem}
\usepackage{inputenc}

\theoremstyle{plain}
\newtheorem{theorem}{Theorem}[section]
\newtheorem{lemma}[theorem]{Lemma}

\theoremstyle{definition}

\newtheorem{remark}[theorem]{Remark}
\newtheorem*{remark*}{Remark}

\numberwithin{equation}{section}

\newcommand\y[1]{#1}

\begin{document}

\title{Oscillations for order statistics of some discrete processes} 

\author[A.~Ottolini]{Andrea Ottolini}
\address[A.~Ottolini]{Department of Mathematics, Stanford University \\ 450 Jane Stanford Way, Stanford CA 94305 (USA).}
\email{ottolini@stanford.edu}

\begin{abstract}
Suppose $k$ balls are dropped into $n$ boxes independently with uniform probability, where $n, k$ are large with ratio approximately equal to some positive real $\lambda$. The maximum box count has a counterintuitive behavior: first of all, with high probability it takes at most two values $m_n$ or $m_n+1$, where $m_n$ is roughly $\frac{\ln n}{\ln \ln n}$. Moreover, it oscillates between these two values with an unusual periodicity. In order to prove this statement and various generalizations, it is first shown that for $X_1,...,X_n$ independent and identically distributed discrete random variables with common distribution $F$, under mild conditions, the limiting distribution of their maximum oscillates in three possible families, depending on the tail of the distribution. The result stated at the beginning follows from the equivalence of ensemble for the order statistics in various allocations problems, obtained via conditioning limit theory. Results about the number of ties for the maximum, as well as applications, are also provided. 
\end{abstract}


\maketitle

\section{Introduction}
\subsection{Extreme value theory}
Even though outliers are often disregarded in statistical models, an understanding of rare and extreme events plays a central role in a variety of situations. An important example, which is analyzed in more detail later, is the occurrence of coincidences for big earthquakes. \\ \\
It is well known (see \cite{Gnedenko}) that for independent identically distributed (i.i.d.) random variables $X_1,...,X_n$ with common distribution function $F(x)$, in order for a law of large numbers for $X_{(n)}=\max_{1\leq i\leq n} X_i$ to hold, it is necessary and sufficient for the $X_i$'s to have a slowly varying tail. More precisely, 
\begin{equation}\label{gne}
\exists \, \,m_n \,\text{s.t.} \, \,X_{(n)}-m_n\, \,\rightarrow 0 \,\,\text{in probability} \Leftrightarrow \lim_{x\rightarrow +\infty }\frac{1-F(x+y)}{1-F(x)}= 0, \forall y>0.
\end{equation}
In case the condition above holds, necessary and sufficient conditions for the existence of a limiting distribution for $X_{(n)}$, after rescaling, are also standard in the literature, and the limits have been widely studied. For a survey on the subject, as well as generalizations and applications, the reader is referred to \cite{Galambos}, \cite{dH}.\\ \\
It is worth mentioning that \eqref{gne} fails to capture the maxima of a variety of distributions: in particular, if the $X_i$'s only take integer values, the above condition cannot be satisfied, since for $0<y<1$ and $x\in \mathbb N$ the above limit is one. The goal of this paper is to investigate the limiting behaviour of $X_{(n)}$ when condition \eqref{gne} is not satisfied. \\ \\
In this paper, `clustering' refers to the extent to which $X_{(n)}$ fails to satisfy a law of large numbers. Roughly speaking, it will be shown that the decay of the mass function determines the size of the cluster. As an instance, for Poisson random variables (whose mass function decays faster than geometrically) $X_{(n)}$ clusters at two values with high probability, while for negative binomial random variables (whose mass function has a geometric decay) $X_{(n)}$ is spread onto all integers with high probability.\\ \\
The first rigorous result in this direction is due to Anderson \cite{Anderson}. He classifies the cluster size for maxima of discrete random variables in terms of their tails. A further analysis is carried out in \cite{Sethuraman}, where lower order statistics are taken into account, as well as in \cite{Athreya}, where the authors study the number of ties.
 \\ \\ The first result of this paper completes the discussion given in Anderson \cite{Anderson}.
\begin{theorem}\label{generalversion}
Let $X_1,...,X_n$ be i.i.d. discrete random variables with common distribution $F$. Suppose the support of $F$ is not bounded from above, and that for some $\gamma\in[0,1]$
\begin{equation}\label{onlyass}
\lim_{n\rightarrow +\infty}\frac{1-F(n+1)}{1-F(n)}=\gamma.
\end{equation}
Then there exist two sequences $\{m_n\}\subset \mathbb N, \{p_n\}\subset [0,1]$ such that
\begin{itemize}
\item $\gamma=0 \Longrightarrow \mathbb P(X_{(n)}=m_n)\sim p_n, \,\mathbb P(X_{(n)}= m_n+1)\sim 1-p_n$,
\item $\gamma\in (0,1) \Longrightarrow \text{ for all $x\in\mathbb Z$, }\mathbb P(X_{(n)}\leq m_n+x)\sim p_n^{\gamma^x}$,
\item $\gamma=1 \Longrightarrow \text{ for all $x\in\mathbb Z$, }\mathbb P(X_{(n)}\leq m_n+x)\sim p_n.$
\end{itemize} 

In the first case, there exists an increasing sequence $\{N_i\}_{i\in \mathbb N}\subset \mathbb N$, with $N_i-N_{i-1}\rightarrow \infty$, such that for $n\rightarrow \infty, n\not\in \{N_i\}$, one has $p_{n+1}\leq p_n$ and $p_{n+1}-p_n \rightarrow 0$. 
\end{theorem}
\begin{remark}

The last part of Theorem \ref{generalversion} is from where the expression \y{``}oscillations" originates. Indeed, a more informal way of interpreting the result is the following: if the endpoints of $[0,1]$ are identified to obtain a circle (and the orientation on $[0,1]$ induces a counterclockwise orientation on the circle), then $p_n=e^{-2\pi i q_n}$, where the sequence $\{q_n\}$ is increasing and satisfies $q_n-q_{n-1}\rightarrow 0$, $\limsup q_n=+\infty$. In words, under this identification the sequence of $p_n$'s will move clockwise on the circle with smaller and smaller steps, winding around the origin infinitely many times. \\ From a probabilistic point of view, in the case $\gamma=0$ the sequences $\{p_n\}, \{1-p_n\}$ represent, for $n$ large, the relative frequency of $X_{(n)}=m_n, X_{(n)}=m_{n+1}$ respectively. Therefore, a histogram of many samples from  $X_{(n)}-m_n$ will not converge to a given shape: it will instead consist of two adjacent columns whose heights oscillate between $0$ and $1$. Moreover, for every fixed $p\in [0,1]$, it is possible to find a subsequence along which the height of the left column will converge to $p$ (and correspondingly, the height of the right one will converge to $1-p$). This again justifies the term \y{``}oscillations."
\end{remark}
\centerline{}
Another natural question concerns the number of times the maximum is expected to occur in a sequence of independent and identically distributed samples from a discrete distribution. This question in only addressed here in the case $\gamma=0$, where the result is the following.
\begin{theorem}\label{numties}
Let $X_1,...,X_n$ be i.i.d. with common distribution $F$ such that
\begin{align*}
\lim_{n\rightarrow +\infty}\frac{1-F(n+1)}{1-F(n)}=0.
\end{align*}
Then, for $p_n, m_n$ as in Theorem \ref{generalversion},
\begin{equation}\label{tiesmaxima}
\mathbb P(\text{at least $k$ ties for the maximum})\sim p_n+1-p_n\sum_{j=0}^{k}\frac{\ln^j\Big(\frac{1}{p_n}\Big)}{j!}.
\end{equation}
\end{theorem} 
As a byproduct, the probability of having no ties is given by $-p_n\ln p_n$, which thus oscillates  between $0$ and $\frac{1}{e}$. \\ \\ In the proof of the theorem, it will be clear that in order to determine the number of ties, one first flips a $p_n$-\y{biased} coin to determine whether the maximum $X_{(n)}$ will assume the value $m_n$ or $m_{n}+1$: if the former happens, then one expects, for all fixed $k$, to have at least $k$ ties with high probability as $n$ gets larger; if the latter happens, the number of ties is Poisson distributed with parameter $\ln(\frac{1}{p_n})$. In the case where $p_n = 0$ or $p_n = 1$ for some $n$, the right side in \eqref{tiesmaxima} equals one (according to the convention $0\ln 0:=0$). Therefore, along subsequences of $\{p_n\}$ converging to $0$ (resp. 1), the $k$-th order statistic $X_{(n-k)}$ for any fixed $k$ will be $m_n$ (resp. $m_n + 1$) with high probability, so that an arbitrarily large number of ties is expected.
\\ \\
Now, suppose that the $p_n$-biased coin gives $X_{(n)}=m_n$. It is interesting to determine how many ties are expected to occur, letting $k$ grow with $n$. This is answered by the following.
\begin{theorem}\label{lotofties}
Let $X_1,...,X_n$ be as before, and let $c>0$ be fixed. Then, there exists a sequence $z_n$ $\{z_n\}$ such that 
\begin{itemize}
\item if $c < 1$, $\mathbb P(X_{n-cz_n}>m_n-1)\rightarrow 1$,
\item if $c > 1$, $\mathbb P(X_{n-cz_n}>m_n-1)\rightarrow 0$.
\end{itemize}
\end{theorem}
In other words, there is a phase transition for the number of ties in the top position, the critical point being $z_n$. It is worth mentioning that $z_n$ oscillates as well. 
\subsection{Multinomial allocations and \y{their} Bayesian counterpart\y{s}}
In order to adapt all previous results to the case of dependent random variables, one approach is to use Poissonization. This standard technique has been widely exploited in a variety of situations: various combinatorial problems in probability \cite{CDM}, cycle structure \cite{Shepp} and longest increasing subsequence \cite{AD} of random permutations, equivalence of ensemble in statistical mechanics \cite{Z}, \y{and many others}.\\ \\
When the randomization leads to an exponential family, from which the original distribution can be obtained by means of conditioning with respect to a sufficient statistic, the results belong to conditioning limit theory. The case where the sufficient statistic is given by $\sum i x_i$ is treated in \cite{ABT}, while here the focus will be on the sum, already considered in \cite{DH}.
\\ \\
In the following, two different allocation problems are considered. Suppose $k$ balls are dropped into $n$ boxes, and we denote by $\underline Y=(Y_1,...,Y_n)$ the box counts. If different balls are dropped independently, $\underline Y$ has a multinomial distribution; if the probabilities of falling into a certain box are unknown, it is natural to consider a Dirichlet mixture of multinomial distributions. \\ \\
The starting point is the conditional representation
\begin{equation}
\mathbb P(\max (Y_1,...,Y_n)\in A)=\mathbb P\Bigl(\max (X_1,...,X_n)\in A|\sum X_i=k\Bigr).
\end{equation}
In the case $Y$ is multinomial with parameter\y{s} $(k,p_1,...,p_n)$, the $X_i$'s can be chosen in such a way that $X_1\sim \mathrm{Poi}(\lambda p_1),..., X_n\sim \mathrm{Poi}(\lambda p_n)$ for every $\lambda$. If the $Y$ are a mixture of multinomial distributions with symmetric Dirichlet kernel with hyperparameter $r$, then the $X_i$'s can be chosen to be in the negative binomial family with parameter $\mathrm{NB(r,p)}$. Bayes' theorem then implies 
\begin{equation}\label{Bayestrick}
\mathbb P(\max (Y_1,...,Y_n)\in A)=\mathbb P(\max (X_1,...,X_n)\in A)\frac{\mathbb P(\sum \tilde X_i=k)}{\mathbb P(\sum X_i=k)},
\end{equation}
where $\tilde X_i$ is the law of $X_i$ conditioned on $\max(X_1,...,X_n)\in A$. Notice that all the three terms \y{on} the right side only involve independent random variables.\y{ }The idea is that, under the appropriate choice of the $X_i$'s on the right hand side, the ratio appearing there is approximately one, so that 
\begin{align*}
\mathbb P(\max (Y_1,...,Y_n)\in A)\sim \mathbb P(\max (X_1,...,X_n)\in A).
\end{align*}
Notice that in general this does not require either side to have a limit: the only important aspect is that the distribution of $\max (X_1,...,X_n)$ and $\max (Y_1,...,Y_n)$ merge together, and then everything about the former (e.g., convergence, limit points etc.) carries over to the latter. For different notions of merging, corresponding to different notions of distance between $f_n(X_1,...,X_n)$ and $f_n(Y_1,...,Y_n)$, the reader is referred to \cite{ADF}. \\ \\
The main results of this paper for the allocation models are the following.
\begin{theorem}\label{multinomial}
Let $(Y_1,...,Y_n)$ be multinomial with parameters $(k,\frac{1}{n},...,\frac{1}{n})$. Suppose that $\lambda=\frac{k}{n}$ is fixed. If $m_n, p_n$ are defined as in Theorem \ref{generalversion} for $F$ the distribution function of a Poisson with parameter $\lambda$, then one has
\begin{align*}
\mathbb P(Y_{(n)}=m_n)\sim p_n, \quad \mathbb P(Y_{(n)}=m_n+1)\sim 1-p_n.
\end{align*}
\end{theorem} 
\begin{theorem}\label{multinomialties}
Let $(Y_1,....,Y_n)$ be multinomial with parameters $(k,\frac{1}{n},...,\frac{1}{n})$. \y{Let} $\lambda=\frac{k}{n}$ be fixed. If $m_n, p_n, z_n$ are defined as in Theorems \ref{numties}, \ref{lotofties}, for $F$ the distribution function of a Poisson with parameter $\lambda$, then one has: 
\begin{itemize}
\item As $n\rightarrow +\infty$, 
\begin{align*}
 \mathbb P(\text{at least $t$ ties at $Y_{(n)}$})\sim p_n+1-p_n\sum_{j=0}^t \frac{\ln^j\Big(\frac{1}{p_n}\Big)}{j!}.
\end{align*}
\item As $n\rightarrow +\infty$, 
\begin{align*}
&\text{if } c<1, \quad \mathbb P(X_{n-cz_n}>m_n-1)\rightarrow 1, \\ \\&\text{if } c>1, \quad \mathbb P(X_{n-cz_n}>m_n-1)\rightarrow 0.
\end{align*}
\end{itemize}
\end{theorem}

\begin{theorem}\label{mixturedirichletoscillations}
Let $(Y_1,...,Y_n)$ be a (symmetric) Dirichlet mixture of multinomials with parameters $(k,r)$. Let $n,k$ be chosen in such a way that for some fixed $p$, $\frac{rp}{1-p}=\frac{k}{n}$ is fixed. If $m_n, p_n$ are defined as in Theorem \ref{generalversion} for $F$ the distribution function of a negative binomial random variable with parameters $(r,p)$, then for every $x\in \mathbb Z$ 
\begin{align*}
\mathbb P(Y_{(n)}\leq m_n+x)\sim p_n^{\gamma^x}.
\end{align*}
\end{theorem}
Borrowing language from statistical mechanics, in the multinomial allocations problem (as well as in its Bayesian version), several features involving the top order statistics can be equivalently studied in the microcanonical and canonical picture \cite{Maxentropy}.
\section{The independent case}
Let $X_1,..., X_n$ be i.i.d. random variables with common distribution $F$, satisfying assumption \eqref{onlyass}. Denote by $\mathcal F:=1-F$ the tail of the distribution. A real extension $\mathcal G$ of $\mathcal F$ will be considered, with the properties
\begin{itemize}
\item $\mathcal G(n)=\mathcal F(n)$ for $n\in \mathbb Z$.
\item $\mathcal G$ is continuous.
\item $\mathcal G$ is decreasing.
\item $\mathcal G$ is log-convex.
\end{itemize}
Such an extension always exists: for example, the log-linear one provided by Anderson in \cite{Anderson} works (yet, sometimes, this is not the most natural one, as in the Poisson case). Notice that the assumption of log-convexity ensures, for example, the existence of
\begin{align*}
\lim_{x\rightarrow +\infty}\frac{\mathcal G(x+\frac{1}{2})}{\mathcal G(x)},
\end{align*}
since the function $x\rightarrow \frac{\mathcal G(x+\frac{1}{2})}{\mathcal G(x)}$ is decreasing. Combined with \eqref{onlyass}, one obtains
\begin{equation}\label{radicesimple}
\lim_{x\rightarrow +\infty}\frac{\mathcal G(x+\frac{1}{2})}{\mathcal G(x)}=\sqrt{\gamma}.
\end{equation}
In a similar way, because $\mathcal G$ is continuous and log-convex, for every $\epsilon\in \mathbb (0,1)$ one has
\begin{equation}\label{radice}
\lim_{x\rightarrow +\infty}\frac{\mathcal G(x+\epsilon)}{\mathcal G(x)}=\gamma^{\epsilon}.
\end{equation}
Let $x_n$ be a solution to $\mathcal G(x_n)=\frac{1}{n}$. Owing to the continuity and monotonicity of $\mathcal G$, such a solution exists and it is unique. Set $m_n$ to be the floor of $x_n+\frac{1}{2}$. By definition, $m_n\in [x_n-\frac{1}{2},x_n+\frac{1}{2}]$. Define also,
\begin{equation}\label{definitionpn}
\theta_n:=\frac{\mathcal G(m_n)}{\mathcal G(x_n)}, \quad p_n:=e^{-\theta_n}.
\end{equation}
Finally, let 
\begin{align*}
z_n:=-\ln \Bigg(\frac{\mathcal G(m_n-1)}{\mathcal G(x_n)}\Bigg).
\end{align*}
\begin{remark}
While the above definitions of $m_n, x_n, p_n, z_n$ depend on the particular choice of $\mathcal G$, all the results stated in the above theorems are equivalent for all such choices. \\ Consider the case $\gamma\in (0,1)$ in Theorem \ref{generalversion}, the other results being analogous. Let $\mathcal G$ and $\mathcal{\tilde G}$ be two different extensions of $\mathcal F$, and let $\tilde x_n, \tilde m_n, \tilde p_n$ be the quantities corresponding to $x_n, m_n, p_n$, obtained by using the extension $\mathcal{\tilde G}$ instead.
The first claim is that 
\begin{align*}
x_n-\tilde x_n\rightarrow 0.
\end{align*}
Indeed, first notice that the floor of both $x_n$ and $\tilde x_n$ is the largest integer $t_n$ such that $\mathcal F(t_n)\geq \frac{1}{n}$. Therefore, $x_n=t_n+\epsilon_n$ and $\tilde x_n=t_n+\tilde \epsilon_n$, with $\epsilon_n, \tilde \epsilon_n\in [0,1)$. Suppose, toward contradiction, that along some subsequence $|\epsilon_n-\tilde \epsilon_n|\sim \epsilon>0$ for some fixed $\epsilon$. In this case, applying \eqref{radice} leads to 
\begin{align*}
\mathcal F(t_n)\gamma^{\epsilon_n}\sim \mathcal G(x_n)=\frac{1}{n}=\mathcal {\tilde G}(\tilde x_n)\sim \mathcal F(t_n)\gamma^{\tilde \epsilon_n},
\end{align*}
which in turn implies $\gamma^{\epsilon}=1$ in the limit, a contradiction. \\ \\ Therefore, along subsequences of $\{x_n\}$ (and thus $\{\tilde x_n\}$, since $x_n-\tilde x_n\rightarrow 0$) that are bounded away from half-integers, the definition of $m_n$ and $\tilde m_n$ will eventually be the same, and consequently $p_n$ will be eventually the same (since $p_n$ depends on $\mathcal G$ only via $m_n$) regardless of the choice of $\mathcal G$. \\  Along subsequences of $\{x_n\}$ (and $\{\tilde x_n\}$) for which $x_n-\left \lfloor{x_n}\right \rfloor \rightarrow \frac{1}{2}$, it may be the case that $m_n$ and $\tilde m_n$, defined from $x_n$ and $\tilde x_n$ respectively, differ by one. \\ 
Without loss of generality, assume that along some subsequence $m_n-x_n\rightarrow -\frac{1}{2}$ and $\tilde m_n-\tilde x_n\rightarrow \frac{1}{2}$. In this case, combining \eqref{definitionpn}, \eqref{radicesimple}, one obtains 
\begin{align*}
p_n\sim e^{-\frac{1}{\sqrt{\gamma}}}, \quad \tilde p_n\sim e^{-\sqrt{\gamma}}.
\end{align*}
In particular, the use of the extension $\mathcal {\tilde G}$ leads to the same asymptotic obtained by using $\mathcal G$, since  
\begin{align*}
\mathbb P(X_{(n)}\leq m_n+x)=\mathbb P(X_{(n)}\leq \tilde m_n-1+x)\sim {(\tilde p_n)}^{\gamma^{-1+x}}\sim e^{-\gamma^{-\frac{1}{2}+x}}\sim p_n^{\gamma^x}.
\end{align*}
\begin{remark}
The freedom in the choice of $\mathcal G$ is not merely an abstract curiosity. As previously mentioned, there are cases (e.g., Poisson distribution) where a certain $\mathcal G$ can be obtained by appropriately replacing a sum with an integral (in the Poisson distribution case, an incomplete gamma function), and such $\mathcal G$ can be easily checked to be log-convex (for classical discrete distribution, this often boils down to the log-convexity of the gamma function). In those cases, it is much easier to use such extensions in order to obtain numerical approximations for $m_n$ and $p_n$, rather than using the artificial log-linear extension introduced by Anderson \cite{Anderson}. 
\end{remark} 
\end{remark}
\subsection{Proof of Theorem \ref{generalversion}}
\begin{proof}[Proof of Theorem \ref{generalversion}]
Owing to the i.i.d. assumption, for all $x\in\mathbb Z$
\begin{align*}
\mathbb P(X_{(n)}\leq x)&=\Bigl(1-\mathcal F(x)\Bigr)^n.
\end{align*}
In the following, note that for $z>-1$,
\begin{equation}\label{logineq}
\frac{z}{1+z}\leq \ln(1+z)\leq z.
\end{equation}
Thus, for every $x\in \mathbb Z$, 
\begin{equation}\label{Bounds}
\mathbb P(X_{(n)}\leq x)=\Bigl(1-\mathcal F(x)\Bigr)^n=\Bigl(1-\frac{\mathcal F(x)}{\mathcal G(x_n)}\frac{1}{n}\Bigr)^n\in \Bigl[e^{-\frac{\mathcal F(x)}{F(x)\mathcal G(x_n)}}, e^{-\frac{\mathcal F(x)}{\mathcal G(x_n)}}\Bigr],
\end{equation}
where the definition of $\mathcal G(x_n)=\frac{1}{n}$ is used in the second equality, while the bounds are derived from \eqref{logineq} applied to $z=-\frac{\mathcal F(x)}{n\mathcal G(x_n)}$ (so that $1+z=F(x)$).\\ \\
It is convenient to split the proof in the two cases $\gamma=0, \gamma>0$:
\begin{itemize}
\item Case $\gamma=0$: the choice of $x=m_n-1$ leads to 
\begin{align*}
\mathbb P(X_{(n)}\leq m_n-1)\leq e^{-\frac{\mathcal F(m_n-1)}{\mathcal G(x_n)}}\leq e^{-\frac{\mathcal G(x_n-\frac{1}{2})}{\mathcal G(x_n)}}\rightarrow 0
\end{align*}
where the first inequality follows from the upper bound in \eqref{Bounds} and the second follows from the both the monotonicity of $\mathcal G$ and that $m_n\in [x_n-\frac{1}{2}, x_n+\frac{1}{2}]$. The last step is a consequence of the assumption $\gamma=0$ and equation \eqref{radice}. If $x=m_n+1$, then the lower bound in \eqref{Bounds} is attained and the very same argument leads to
\begin{align*}
\mathbb P(X_{(n)}\leq m_n+1)&\geq \exp\Big\{-\frac{\mathcal G(x_n+\frac{1}{2})}{\mathcal G(x_n)F(m_n+1)}\Big\}\\&\rightarrow 1.
\end{align*}
This result proves the clustering effect on the two values $m_n, m_n+1$. In general, the same computations lead to
\begin{align*}
\mathbb P(X_{(n)}\leq m_n)\sim \exp\Big\{-\frac{\mathcal G(m_n)}{\mathcal G(x_n)}\Big\}=p_n,
\end{align*}
from which the result
\begin{align*}
\mathbb P(X_{(n)}\leq m_n)\sim p_n, \quad \mathbb P(X_{(n)}=m_n+1)\sim 1-p_n,
\end{align*} 
follows. As for the last statement in the Theorem, notice that by the definition of $x_n$
\begin{align*}
\frac{\mathcal G(x_{n+1})}{\mathcal G(x_{n})}=\frac{\frac{1}{n+1}}{\frac{1}{n}}\rightarrow 1.
\end{align*}
Suppose, toward contradiction, that $x_{n+1}-x_{n}\rightarrow \epsilon>0$ along some subsequence. Owing to $\eqref{radice}$, the limit $\ell:=\lim_{n\rightarrow +\infty}\frac{\mathcal G(x_n+\epsilon)}{\mathcal G(x_n)}$ exists and it is equal to zero (thanks to the assumption $\gamma=0$). Thus, necessarily $x_{n+1}-x_{n}\rightarrow 0$. By continuity of $\mathcal G$,
\begin{align*}
\mathcal G(x_n)-\mathcal G(x_{n+1})\rightarrow 0.
\end{align*}
Define $N_i$ as the increasing sequence of natural numbers for which $x_{N_i}\leq i+\frac{1}{2}, x_{N_i+1}>i+\frac{1}{2}$ for all integers $i$. Consider $x_n$ for $n\not \in\{N_i\}_{i\in \mathbb N}$, and recall that $m_n$ is the floor of $x_n+\frac{1}{2}$. For such $n$'s, one has $m_n=m_{n+1}$, and consequently $p_{n+1}\leq p_n$ (owing to the monotonicity of $\mathcal G$ and \eqref{definitionpn}), and $p_{n+1}-p_n\rightarrow 0$ (owing to the continuity of $\mathcal G$). Finally, notice that $N_{i+1}-N_i\rightarrow \infty$ since $x_n\rightarrow +\infty$, $x_{n+1}-x_n\rightarrow 0$, which concludes the proof. 
\item Case $\gamma>0$: for all fixed $x\in \mathbb Z$ one has, owing to \eqref{Bounds} and \eqref{onlyass}, 
\begin{align*}
\frac{\mathcal F(m_n+x)}{\mathcal G(x_n)}\sim \frac{\mathcal F(m_n+x)}{\mathcal G(x_n)F(m_n+x)}\sim \frac{\mathcal G(m_n)\gamma^x}{\mathcal G(x_n)}.
\end{align*}
Therefore, in this case
\begin{align*}
\mathbb P(X_{(n)}\leq m_n+x)&= \Bigl(1-\mathcal F(m_n+x)\Bigr)^n\\&\sim \exp\Big\{-\frac{\mathcal G(m_n)}{\mathcal G(x_n)}\gamma^x\Big\}\\&= p_n^{\gamma^x},
\end{align*}
which concludes the proof in the case $\gamma>0$ (notice that, if $\gamma=1$, one has $\gamma^x\equiv 1$).
\end{itemize}
\end{proof}
\subsection{Proof of Theorems \ref{numties} and \ref{lotofties}}
\begin{proof}[Proof of Theorem \ref{numties}]
First of all, notice that the assumption $\gamma=0$ guarantees that $z_n\rightarrow +\infty$. Moreover, since by definition
\begin{align*}
n\mathcal F(m_n-1)=z_n, 
\end{align*} 
and $m_n\rightarrow +\infty$, \y{it follows that} $z_n=o(n)$. 
Now, consider the binomial formula for the order statistics 
\begin{equation}\label{binomialformula}
\mathbb P(X_{(n-k)}\leq x)=\sum_{j=0}^{k}{{n}\choose {j}}[F(x)]^{n-j}[1-F(x)]^{j}.
\end{equation}
If $x=m_{n}-1$, $j$ is fixed and $n\rightarrow +\infty$, one has
\begin{align*}
{n\choose j}\Big(1-\frac{z_n}{n}\Big)^{n-j}\Big(\frac{z_n}{n}\Big)^{j}\sim \frac{z_n^j}{j!}e^{-z_n}\rightarrow 0.
\end{align*}
When $k$ is fixed, since there are only finitely many terms in \eqref{binomialformula} and each of these converges to $0$, one concludes 
\begin{equation}\label{clusterties}
\mathbb P(X_{(n-k)}\leq m_n-1)\rightarrow 0.
\end{equation}
Now fix $p\in (0,1)$, and look at a subsequence $p_{n}\rightarrow p$ (where for simplicity the $p_{n_k}$ subsequence was renamed $p_n$). Then, since $\mathbb P(X_{(n)}>m_{n}+1)\rightarrow 0$, along this subsequence
\begin{align*}
\mathbb P(X_{(n-k)}=m_{n}+1)&\sim \mathbb P(X_{(n-k)}>m_{n})\\&=1-\sum_{j=0}^{k}{n\choose j}\Big(1-\frac{\theta_n}{n}\Big)^{n-j}\Big(\frac{\theta_n}{n}\Big)^{j}\\&\rightarrow p\Big(\frac{1}{p}-\sum_{j=0}^{k} \frac{\ln^j(\frac{1}{p})}{j!}\Big), 
\end{align*}
where it was used that $\mathcal F(m_n)=\frac{\mathcal F(m_n)}{\mathcal G(x_n)}\mathcal G(x_n)=\frac{\mathcal G(m_n)}{\mathcal G(x_n)}\Big(\frac{1}{n}\Big)=\frac{\theta_n}{n}$.

\end{proof}

Before moving \y{on} to the proof of Theorem \ref{lotofties},  recall the incomplete gamma function
\begin{align*}
\Gamma(k,z)=\int_0^{z}t^{k-1}e^tdt.
\end{align*}
Notice that for $k$ \y{an} integer, integration by parts shows \y{that}
\begin{equation}\label{twopoisson}
e^{-z}\sum_{j=k}^{+\infty}\frac{z^j}{j!}=\frac{\Gamma(k+1,z)}{\Gamma(k+1)}.
\end{equation}
To find asymptotics for $\Gamma(k,z)$, it is useful to recall the Laplace asymptotic formula (see, e.g., Theorem 3.5.3 in \cite{AGZ}):
\begin{theorem}[Laplace asymptotic formula]\label{laplace}
Let $S(x)$ be a smooth function on $(a,b)$. Then 
\begin{itemize}
\item If $S'(x)<0$ for all $x\in (a,b)$, then 
\begin{align*}
\int_a^be^{-mS(x)}f(x)dx=\frac{1}{m}\frac{1}{S'(b)}f(b)e^{-mS(b)}\Big(1+O\Big(\frac{1}{m}\Big)\Big),\quad  m \rightarrow +\infty.
\end{align*}
\item If $S$ has a unique nondegenerate minimum $x_0$ in $(a,b)$, then
\begin{align*}
\int_a^be^{-mS(x)}f(x)dx=\sqrt{\frac{2\pi}{mS''(x_0)}}f(x_0)e^{-mS(x_0)}\Big(1+O\Big(\frac{1}{m}\Big)\Big),\quad m\rightarrow +\infty.
\end{align*}
\end{itemize}
\end{theorem}

\begin{proof}[Proof of Theorem \ref{lotofties}]
Using \eqref{binomialformula} 
\begin{align*}
\mathbb P(X_{(n-cz_n+1)}>m_{n}-1)&=1-\sum_{j=0}^{cz_n-1}{n\choose j}[F(m_n-1)]^{n-j}[1-F(m_n-1)]^{j}\\&=\sum_{j=cz_n}^{n}{n\choose j}\Big(1-\frac{z_n(1+o(1))}{n}\Big)^{n-j}\Big(\frac{z_n(1+o(1))}{n}\Big)^j.
\end{align*}
Now, given $c\in (0,+\infty)$, consider $m=m(c)$ large (to be fixed later). The sum above can be split into
\begin{align*}
\mathbb P(X_{(n-cz_n+1)}>m_{n}-1)&= \sum_{j=cz_n}^{mcz_n-1}{n\choose j}\Big(1-\frac{z_n(1+o(1))}{n}\Big)^{n-j}\Big(\frac{z_n(1+o(1))}{n}\Big)^j\\&+\sum_{j=mcz_n}^{n}{n\choose j}\Big(1-\frac{z_n(1+o(1))}{n}\Big)^{n-j}\Big(\frac{z_n(1+o(1))}{n}\Big)^j\\&=: A+B
\end{align*}
First, consider the second summand: since ${n\choose j}\leq \frac{n^j}{j!}$, each term can be bounded: \begin{align*}
{n\choose j}\Big(1-\frac{z_n(1+o(1))}{n}\Big)^{n-j}\Big(\frac{z_n(1+o(1))}{n}\Big)^j\leq \frac{z_n^j}{j!}\rightarrow 0.
\end{align*}
By choosing $m$ large enough that $mc>e$, using 
\begin{align*}
\frac{z_{n}^{j+1}}{(j+1)!}\leq \frac{z_n^{j}}{j!}\frac{z_n}{j}\leq \frac{z_n^j}{j!}\frac{1}{mc},
\end{align*}
$B$ can be bounded by a geometric series. Therefore, using crude bounds with Stirling's approximation, 
\begin{align*}
B\leq \sum_{j=mcz_n}^{n}\frac{z_n^j}{j!}\leq \frac{mc}{mc-1}\frac{z_n^{mcz_n}}{(mcz_n)!}\leq \frac{mc}{mc-1} \Big(\frac{e}{mc}\Big)^{mcz_n}\rightarrow 0.
\end{align*}
Going back to the first summand, to \y{finish} the proof it is enough to show that 
\begin{align*}
A=\sum_{j=cz_n}^{mcz_n-1}{n\choose j}\Big(1-\frac{z_n(1+o(1))}{n}\Big)^{n-j}\Big(\frac{z_n(1+o(1))}{n}\Big)^j
\end{align*}
converges to $0$ if $c > 1$ and converges to $1$ if $c < 1$, regardless of $m$. \y{In} this \y{regime}, $j\rightarrow +\infty, j=o(n)$, \y{so} 
\begin{align*}
A\sim \sum_{j=cz_n}^{mcz_n-1}\frac{e^{-z_n}z_n^j}{j!}\sim \sum_{cz_n}^{+\infty}\frac{e^{-z_n}z_n^{j}}{j!},
\end{align*}
where the last step follows from the fact that $B\rightarrow 0$. Using \eqref{twopoisson}, 
\begin{align*}
A\sim \frac{\Gamma(cz_n+1,z_n)}{\Gamma(k_n(c))}=\frac{\int_0^{z_n}t^{cz_n}e^{-t}dt}{(cz_n)!}.
\end{align*}
Changing the variable $t=cz_ns$ and using Stirling\y{'s} approximation, \y{we obtain}
\begin{align*}
A&\sim \frac{1}{(cz_n)!}\int_0^{\frac{1}{c}}cz_n\exp\Big(-scz_n+cz_n\ln(cz_n))+cz_n\ln s\Big)ds\\&\sim \frac{cz_n}{e^{-cz_n}\sqrt{2\pi cz_n}}\int_0^{\frac{1}{c}}e^{-cz_n[s-\ln s]}ds.
\end{align*}
Since the function $S(s)=s-\ln s$ has a global minimum at $s=1$, with $S(0)=1$, $S''(1)=1$, if $c>1$ the first part of Theorem \ref{laplace} gives
\begin{align*}
A\sim \frac{\sqrt{cz_n}}{e^{-cz_n}\sqrt{2\pi}}\frac{1}{1-\frac{1}{c}}e^{(-cz_n)(\frac{1}{c}+\ln c)}\rightarrow 0,
\end{align*}
since $1<\frac{1}{c}+\ln c$. In the case $c>1$, the second part of Theorem \ref{laplace} leads to
\begin{align*}
A\sim \frac{\sqrt{cz_n}}{e^{-cz_n}\sqrt{2\pi}}\frac{\sqrt{2\pi}}{\sqrt{cz_n}}e^{-cz_n}=1,
\end{align*}
which concludes the proof.
\end{proof}
\subsection{Some examples: the Poisson, the negative binomial, and the discrete Cauchy}

Consider the case where $X_1 \sim \mathrm{Poi}(\lambda)$. Anderson (\cite{Anderson}) already proved the result 
\begin{align*}
\mathbb P(X_{(n)}\in \y{\{} m_n,m_{n}+1\y{\}})\rightarrow 1. 
\end{align*}
In the language of this paper, the Poisson distribution falls into the case $\gamma=0$ of Theorem \ref{generalversion}, since
\begin{align*}
\mathcal F(x)\rightarrow e^{-\lambda}\frac{\lambda^{x+1}}{\Gamma(x+2)}.
\end{align*}
Notice that the most natural \y{choice} for $\mathcal G$ in this case is given by the incomplete gamma function, rather than the log-linear extension. Following the proof of \y{Theorem} \ref{generalversion}, it is easy to see that
\begin{align*}
\mathbb P(X_{(n)}\not\in \{m_n,m_{n}+1\})\leq \Big(\frac{\lambda}{x_n+1}\Big)^{m_n-x_n}.
\end{align*}
This bound is important, as it shows the clustering may emerge even for small values of $n$, provided that $\lambda$ is small. As for the value of $m_n$, in \cite{Kimber}, it is shown that in first approximation
\begin{equation}\label{asymptoticuseless}
x_n \sim \frac{\ln n}{\ln \ln n}.
\end{equation}
However, this estimate is extremely poor, as is shown in \cite{Briggs}. In particular, if $W(z)$ is the solution to $e^{W(z)}W(z)=z$ (known as the Lambert function, see \cite{Knuth}), then a much better approximation is given by
\begin{align*}
\tilde x_n= y_n+\frac{\ln \lambda-\lambda-\frac{1}{2}\ln(2\pi)-\frac{3}{2}\ln(y_n)}{\ln(y_n)-\ln \lambda}, \quad y_n=\frac{\ln n}{W\Big(\frac{\ln n}{\lambda e}\Big)}.
\end{align*}
\\ \\
The Negative binomial distribution $N(r,p)$ falls into the second category of Theorem \ref{generalversion}, with $\gamma=p$, since 
\begin{align*}
\frac{\mathcal F(n+1)}{\mathcal F(n)}=\frac{\Gamma(n+2+r)}{\Gamma(n+1)\Gamma(r)}\frac{\Gamma(n+1)\Gamma(r)}{\Gamma(n+1+r)}\frac{\int_0^pt^{n+1}(1-t)^{r-1}dt}{\int_0^pt^{n}(1-t)^{r-1}}.
\end{align*}
and, using \ref{laplace} and the property of the gamma function, it is easy to obtain
\begin{align*}
\frac{\mathcal F(n+1)}{\mathcal F(n)}=\frac{n+1+r}{n+1}p\Big(1+o(\frac{1}{n})\Big)\rightarrow p.
\end{align*}
Finally, the discrete Cauchy distribution falls into the third regime, since in that case
\begin{align*}
\frac{\mathcal F(n+1)}{\mathcal F(n)}=\frac{\frac{1}{1+(n+1)^2}}{\frac{1}{1+n^2}}\rightarrow 1.
\end{align*}

\section{The dependent case}
As explained in the introduction, it is possible to export the previous results to a certain class of allocation problems.
The main ingredient is the local central limit theorem (see, \y{e.g.}, \cite{Gne-Kol})\y{.}
\begin{lemma}[Local Central Limit theorem]\label{LCLT}
Let $X_1,...,X_n$ be discrete i.i.d. random variables, with $\mathbb E(X_1)=\mu, \mathbb Var(X_1)=\sigma^2$, such that the values taken on by $X_1$ are not contained in some infinite \y{progression} $a + q \mathbb Z$ for integers $a,q$ with $q > 1$. Then, for every integer $t$, 
\begin{align*}
\mathbb P\Bigl(\sum_{i=1}^n X_i=t\Bigr)=\frac{1}{\sqrt{2\pi n\sigma}}\exp\Big(-\frac{(t-n\mu)^2}{2n\sigma^2}\Big)+o\Big(\frac{1}{\sqrt{n}}\Big),
\end{align*}
the error being \y{uniform} in $t$.
\end{lemma}
\subsection{Multinomial allocations}
First, consider the case of multinomial allocations, all boxes being equally likely. 
\begin{proof}[Proof of theorem \ref{multinomial}]
Let $X_1,...,X_n$ be i.i.d. with {\tiny $X_1\sim Poi(\lambda)$} $X_1 \sim \mathrm{Poi}(\lambda)$. By means of \eqref{Bayestrick},
\begin{align*}
\mathbb P(Y_{(n)}\in A)&=\mathbb P(X_{(n)}\in A)\frac{\mathbb P(\sum X_i=k, X_{(n)}\in A)}{\mathbb P(\sum X_i=k)}.
\end{align*}
Notice that
\begin{align*}
\mathbb P(X_{(n)}=m_n)\sim p_n, \quad \mathbb P(X_{(n)}=m_{n+1})\sim 1-p_{n}
\end{align*}
owing to Theorem \ref{generalversion}. Moreover, $\sum_{i=1}^n X_i\sim \text{Poi}(k)$, so that 
\begin{equation}
\mathbb P\Big(\sum_{i=1}^n X_i=k\Big)=e^{-k}\frac{k^k}{k!}\sim \frac{1}{\sqrt{2\pi k}}.
\end{equation}
It remains to estimate the \y{``tilded''} version of the $X_i$'s. If $A=\tilde m_n$, where $\tilde m_n=m_n$ or $\tilde m_n=m_n+1$ then $\{\tilde X_i\}_{i=1}^{n-1}=\{X_i\}_{i=1}^n\setminus X_{(n)}$ are still independent \y{and} identically distributed according to 
\begin{align*}
\mathbb P(\tilde X_1=t)=\frac{e^{-\lambda }\frac{\lambda^t}{t!}}{F(\tilde m_n)}, \quad t=0,.., \tilde m_n,
\end{align*}
$F$ being the cumulative distribution of Poisson as in section $3$.
By symmetry, each $X_i$ is equally likely to be the maximum, so that $\mathbb P(X_{(n)}=X_i)=\frac{1}{n}$. Therefore, 
\begin{align*}
\mathbb P\Big(\sum_{i=1}^n X_i=k|X_{(n)}=\tilde m_n \Big)=\mathbb P\Big(\sum_{i=1}^{n-1}\tilde X_i=n-\tilde m_n\Big),
\end{align*}
which can be now estimate\y{d} by means of Theorem \ref{LCLT} (notice that the condition that $X_1$ does not belong to a subprogression is obviously satisfied). The first moment is 
\begin{align*}
\mathbb E(\tilde X_1)=\frac{1}{F(\tilde m_n)}\sum_{j=0}^{\tilde m_n}e^{-\lambda}\frac{j\lambda^j}{j!}=\lambda \frac{F(\tilde m_n-1)}{F(\tilde m_n)}=\lambda(1+o(1)).
\end{align*}
Similarly, the variance is given by 
\begin{align*}
\mathbb Var (\tilde X_1^2)&=\frac{\sum_{j=0}^{m_n} e^{-\lambda}\frac{j^2\lambda^j}{j!}}{F(\tilde m_n)}-\lambda^2(1+o(1))\\&=\frac{\lambda^2 F(\tilde m_n-2)+\lambda F(\tilde m_n-1)}{F(\tilde m_n)}-\lambda^2(1+o(1))\\&=\lambda(1+o(1)).
\end{align*}
Hence, the local central limit theorem leads to 
\begin{align*}
\mathbb P\Big(\sum_{i=1}^{n-1}\tilde X_i= n-\tilde m_n\Big)&=\frac{1}{\sqrt{2\pi(n-1)\lambda}}\exp\Big\{\frac{(k-\tilde m_n-(n-1)\lambda(1+o(1))}{2(n-1)\lambda }\Big\}\\&\sim \frac{1}{\sqrt{2\pi k}}\exp\Big\{\frac{(\lambda-\tilde m_n)^2}{2k}  \Big\}\\&\sim \frac{1}{\sqrt{2\pi k}},
\end{align*}
where the last step follows from $\tilde m_n^2=o(k)$, a consequence of $\tilde m_n\sim \frac{\ln n}{\ln \ln n}$ and $k=\lambda n$. Therefore, 
\begin{align*}
\mathbb P(Y_{(n)}=m_n)\sim \mathbb P(X_{(n)}=m_n)\sim p_n, \quad \mathbb P(Y_{(n)}=m_n+1)\sim 1-p_n,
\end{align*}
as desired.
\end{proof} 
For the proofs of Theorem \ref{multinomialties}, the very same argument can be applied. Indeed, the only difference is that the number of copies of $\tilde X$'s is now $n-t,n-t_n(c)$ respectively. However, this does not affect the central limit theorem, since $t$ and $t_n(c)$ are much smaller than $n$ (so that the asymptotic in the central limit theorem remains the same). 

\subsection{A Bayesian version}
Consider now the Bayesian variant of the multinomial allocation problem. The idea is again the same, but a proof is sketched for the sake of completeness.
\begin{proof}[Proof of Theorem \ref{mixturedirichletoscillations}]
Fix $x\in \mathbb Z$. By means of \eqref{Bayestrick} and the conditional representation of Dirichlet mixture of multinomials as negative binomials, it suffices to show that for $X_1,...,X_n$ i.i.d. with $X_i \sim \mathrm{NB}(r,p)$ and $\frac{rp}{1-p}=\frac{k}{n}$, 
\begin{align*}
\mathbb P\Bigl (\sum_{i=1}^n X_i=k\Bigr )\sim \mathbb P\Bigl (\sum_{i=1}^n X_i=k| X_{(n)}\leq m_n+x\Bigr),
\end{align*}
as $n, k\rightarrow +\infty$. As before, the right hand side can be rewritten as 
\begin{align*}
\mathbb P\Bigl (\sum_{i=1}^n \tilde X_i=k\Bigr),
\end{align*}
where $\tilde X_i$ is the tilded version of $X_i$\y{,} given by
\begin{align*}
\mathbb P(\tilde X_i=s)=\frac{\mathbb P(X_1=s)}{\mathbb P(X_1\leq m_n+x)},\quad s\leq m_n+x. 
\end{align*}
Since both the mean and the variance are asymptotically the same for {\tiny $X_1$} $X_i$ and $\tilde X_i$ (using that $m_n+x\rightarrow +\infty$), the local central limit theorem can be applied to conclude the proof. 
\end{proof}

\section{Numerical results and applications}\label{numerics}
While theoretically satisfactory, the question remains of whether these asymptotic results are of any use in simulations or real models (or whether $n$ has to be unreasonably large for the effect to be manifest). Here are the main take-aways from some simulations for i.i.d. discrete random variables and random allocation models:\\ 
\begin{itemize}[noitemsep, topsep=0pt]
\item The merging of dependent and independent cases works well for reasonable values of $k$ and $n$. If the theory gives good approximations in some regime for the independent random variables, it also works for the dependent ones.
\centerline{}
\item In order to detect the oscillations of the maxima (as well as the other features) in the Poisson case, the quantity $\frac{\lambda}{x_n+1}$ has to be small. Since $x_n$ grows sublogarithmically, $n$ has to be extremely large compared to $\lambda$ (in particular, it is necessary to have $n\gg e^{\lambda}$). This explains why simulations essentially fail for $\lambda \gg 1$, why they work for $\lambda=O(1)$ provided $n$ is large (for $\lambda=1$, in order to obtain $p_n$ within an error of $\epsilon$, it is necessary to have at least $n\geq e^{\frac{1}{\epsilon}}$), and why they are excellent for $\lambda\ll 1$, even with relatively small $n$.
\end{itemize}
\subsection{The role of the mean}
Here are some numerical values for $m_n, x_n, p_n$, depending on $n$. For now, we take $\lambda=1$, (but, as explained above, soon $\lambda$ will be small).\\ 
\begin{center}
     \begin{tabular}{ | l | l | l | p{2cm} |}    \hline
    $n$ & $x_n$ & $m_n$ & $p_n$ \\ \hline
    $10^3$ & 4.63591 & 5 & 0.58694674 \\ \hline
    $10^4$ & 5.84299 & 6 & 0\y{.}47741767 \\ \hline
    $10^5$ & 6.95712 & 7 & 0\y{.}40055502 \\ \hline
    $10^6$ & 8.00608 & 8 & 0\y{.}36296353 \\ \hline
    $10^9$ & 10.89530 & 11 & 0\y{.}46225972 \\ \hline
    $10^9+10^7$ & 10.8993 & 11 & 0.45873497 \\ \hline
    $10^{50}$ & 40\y{.}0255 & 40 & 0\y{.}333090 \\
    \hline
    \end{tabular}
    \medskip
    
  Values of $m_n, x_n, p_n$ from Theorem \ref{generalversion} as a function of $n$ (with $\lambda=1$). The value $x_n$ can be obtained e.g., via approximating the Lambert function as in \cite{Knuth}, or by means of numerical method\y{s}.
  \end{center}
Here are some observations from the table:
\begin{itemize}
\item The value $x_n$ grows slowly. At first sight, it seems logarithmic, as the factor $\ln \ln n$ in the asymptotic \eqref{asymptoticuseless} is hard to detect for reasonable value\y{s} of $n$. Only the last entry gives an insight in this direction,
\item The absence of a law of large numbers, as well as the oscillations, already emerges in this picture: the value of $p_n$ does not exhibit any limiting behavior,
\item The period of the oscillations (i.e., the difference $N_{i+1}-N_i$ in the language of Theorem \ref{generalversion}) is increasing in $n$.

\end{itemize}
 In the following, $1000$ trials of the \y{experiment} \y{``}drop $\lambda n$ balls into $n$ boxes independently" were simulated. Because of Theorem \ref{multinomial}, the maximum box count should be $m_n$ or $m_n+1$ with probabilities given by, respectively, $p_n$ or $1-p_n$. The outcomes are represented in the following table:
\centerline{}
\centerline{}
\begin{center}
    \begin{tabular}{ | l | l | l | p{2cm}| p{2cm} | p{2cm}|}   \hline
    $n$ & $\lambda$ & $m_n$ &$p_n$ & $f_n$ & $o_n$ \\ \hline
    $10^5$ & 0.1 & 3 & 0.675268 & 0.69 & 0.005 \\ \hline
    $10^5$ & 1  & 7 & 0.40055502 &  0.353 & 0.11  \\ \hline
    $10^5$ & 10 & 25 & 0.325168 & 0.162 & 0.467 \\ 
   
    \hline
        \end{tabular}
    \medskip

Comparison between $p_n$ and the relative frequency $f_n$ for $m_n$ out of $1000$ trials of the experiment \y{``}drop $\lambda n$ balls into $n$ boxes". The last column represents the fraction $o_n$ of maxima outside the cluster values $m_n, m_n+1$.
    \end{center}

 \centerline{}
Here are some observations:
\begin{itemize}
\item For large $\lambda$, the approximation is useless. This is not surprising since the quantity $\frac{\lambda}{x_{n}+1}$ is far from being negligible. 
\item For small $\lambda$, the approximation works well\y{,} and the theory can be fully appreciated for reasonable $n$, since the quantity $\frac{\lambda}{x_n+1}$ is small,
\item If $\lambda$ increases, the value of $m_n$ also increases. However, the $\lambda$-correction in $x_n$ (and hence, in $m_n$) is rather small. This is the reason why small $\lambda$ is preferable in order to see the results from the theory,
\item Notice that since $x_n$ grows sublogarithmically, for fixed $\lambda$ we need to significantly increase $n$ to see an improvement. On the other hand, once $\lambda$ is small, the theory works even for $n$ small (e.g., $n=1000$).
\end{itemize}
That being said, the focus will be now on the regime $\lambda=0.01$ in order to even better capture the \y{``}oscillating behavior" of $p_n$. This is an experiment of dropping $\lambda n$ balls into \y{$n$} boxes, for various values of $n$,
\begin{center}
    \begin{tabular}{ | l | l  | p{2cm}| p{2cm} | p{2cm}|}   \hline
    $n$ & $m_n$ &$p_n$ & $f_n$ \\ \hline
    $2000$ & 1  & 0.8902 & 0.9073 \\ \hline
    $4000$ & 1 & 0.8039 & 0.8171 \\ \hline
    $8000$ & 1 & 0.6602 & 0.6646 \\ \hline
    $16000$ & 1 & 0.4492 & 0.4548 \\ \hline
    $32000$ & 1 & 0.2106 & 0.2047 \\ \hline
    $64000$ & 1 & 0.0469 & 0.0357 \\ \hline
    $128000$ & 1 & 0.0023 & 0.0017 \\ \hline
    $256000$ & 1 & 0.0000 & 0,0000\\ \hline
    $512000$ & 2 & 0.9103 & 0.9181 \\
    \hline
        \end{tabular}
    \medskip

Oscillation of $p_n$ for $n=1000\times 2^m$, $m\in \{1,...,9\}$, and $\lambda=0.01$. For each $n$, I simulated $10000$ times the experiment of dropping $\lambda n$ balls into $n$ boxes. As before, $f_n$ denotes the relative frequency of $m_n$.   \end{center}
\centerline{}

The oscillation is visible in the last step: $p_n$ \y{``}refreshes" at $1$ after $x_n-m_n$ changes its sign, a phenomenon that happens on a long scale. \\ \\
Moving to the number of ties, Theorem \ref{numties} implies that the probability of having $t$ ties at the value of the maximal order statistic is given by
\begin{align*}
\mathbb P(\text{$t$ ties for the maximal order statistic})\sim\, p\frac{\ln^{t+1}\Big(\frac{1}{p}\Big)^{t+1}}{(t+1)!}
\end{align*}
Here is a simulation of the process:
\begin{center}
    \begin{tabular}{ | p{2cm}| p{2cm} | p{2cm}|}   \hline
    $t$ & $t_n$ &$f_n$ \\ \hline
    $0$ & 0\y{.}35948  & 0.3613 \\ \hline
    $1$ & 0\y{.}14382 & 0.1431 \\ \hline
    $2$ & 0\y{.}03836 & 0.0385  \\ \hline
    $3$ & 0\y{.}0076 & 0.008  \\ 
    \hline
        \end{tabular}
    \medskip

The result of $10000$ simulation of dropping $160$ balls into $16000$ boxes ($\lambda=0.01$, $p_n=0\y{.}44924115$) and counting the number of ties $t$. The relative frequencies $f_n$ are compared to the theoretical probabilities $t_n$. \end{center}
\centerline{}
The results \y{are very} accurate for small number\y{s} of ties. Finally, here are simulations for the result for the cluster size on the top two spots for the same value\y{s} of $n$ and $\lambda$: the theoretical result is that about $156.65$ boxes should have a count of $1$ or $2$ balls. The average of $10000$ experiments gives the result $159.21$.\\ \\

\subsection{Coincidence for earthquakes}
In the popular imagination, big earthquakes are one of the main instances of randomness in natural events. Heuristically, big earthquakes are not independent of each other (as everyone who lives in a seismic area knows), and they \y{instead} tend to clump together. As such, a reasonable model is that of inter-arrival times (forgetting about any geographic information) which are distributed according a negative binomial (see \cite{Earthquakes}), which is suitable for representing positively correlated events. \\ \\
In the following, we adopt this model and use our theory to study the occurrence of multiple big earthquakes in a given window of time, using data from \cite{USGS}. For instance, can we explain the occurrence of multiple earthquakes in a given hour by purely statistical argument\y{s}, without any \y{``}cause-effect" argument\y{s}?\\ \\
In the language of the previous section, $X_i$ is the number of earthquakes of magnitude above $6$ which occurred in an hour $i$ of the day, with $i$ running from $1$ to $24$. We consider realizations over three periods of time: the $70s$, the $80s$, and the $90s$, which correspond respectively to $3652, 3653,$ and $3652$ instances of the $X_i$'s. 
\begin{center}
    \begin{tabular}{ | l | l | p{3cm}| p{2cm}| p{2cm}| p{2cm}| p{2cm}| p{2cm}|}   \hline
    Decade & $E(X_1)$ &$Var(X_1)$ & $r$ & $p$  \\ \hline
    $70's$ & $0.01226497$  & $6.089 \times 10^{-4}$ &   $0.0496$ &   $0.0472$\\ \hline
    $80's$ & 0.01382425 & $7.112 \times 10^{-4}$  &   $0.0514$   &   $0.0489$\\ \hline
    $90's$ & 0.01669177 & $8.6302 \times 10^{-4}$  & $0.0517$ &   $ 0.0489$ \\ 
    \hline
        \end{tabular}
    \medskip

Occurrence of earthquakes in a given hour across three different decades, with corresponding estimators with a negative binomial model. \end{center}
\centerline{}

We compare the results between our theory, a result of $10^6$ simulations of negative binomial random variables with the same parameter, and the empirical data. We expect the maximum number of earthquakes in a single hour within a day to be either $0$, $1$, or $2$ (theoretically, numerically, and empirically it is almost impossible to observe more than $3$ earthquakes in a given hour). Here are the results:

\begin{center}
    \begin{tabular} { | l | l  | p{4cm}| p{3.5cm} | p{4cm}|} \hline
    Decade & theory & numerics & empirical data \\ \hline
    $70's$ & 75.17  -  23.49  -  1.28 & 75.06  -  24.08  -  0.81 &   76.01  -  23.11  -  0.84 \\ \hline
    $80's$ & 72.52  -  25-92  -  1.48 & 72.34  -  26.65  -  0.97  &  74.24  -  24.89  -  0.95 \\ \hline
    $90's$ & 67.88  -  30.24  -  1.79 & 67.65  -  31.07  -  1.22 & 69.52  -  29.30  -  1.01 \\ 
    \hline
        \end{tabular}
    \medskip

Maximum number of earthquakes in a single hour within a day. The notation $a$ - $b$ - $c$ denotes the percentages of days with $0$, $1$, or $2$ as a maximum. \end{center}
\centerline{}
\section*{Acknowledgement}
The author wishes to thank Persi Diaconis for suggesting the problem, and for many helpful discussions on the subject. The author is also indebted to Daniel Dore and two referees for their careful revision of the first drafts.

\end{document}